\documentclass[11pt]{article}
\usepackage{graphicx}
\usepackage{epstopdf}
\usepackage{subcaption}
\usepackage{enumerate}
\usepackage{amssymb,a4wide,latexsym,makeidx,epsfig,fleqn}
\usepackage{amsthm}
\usepackage{amsmath}
\allowdisplaybreaks[4]
\usepackage{enumerate}
\usepackage{graphicx}
\usepackage{color}
\usepackage{epstopdf}
\newtheorem{theorem}{Theorem}[section]

\newtheorem{lemma}[theorem]{Lemma}

\newtheorem{problem}[theorem]{Problem}

\begin{document}
	\textwidth 150mm \textheight 230mm
	\setlength{\topmargin}{-15mm}
	\title{Sufficient conditions for fractional $k$-factor-critical graphs with minimum degree to be $k$-factor-critical
		\footnote{This work is supported by the National Natural Science Foundations of China (Nos. 12371348, 12201258), the Postgraduate Research \& Practice Innovation Program of Jiangsu Normal University (No. 2025XKT0633), High-Quality Science and Technology Cultivation Project of Jiangsu Normal University (No. JSNUGZL2026069).}}
	\author{{ Jiaxu Zhong, Yong Lu\footnote{Corresponding author}}\\
		{\small  School of Mathematics and Statistics, Jiangsu Normal University,}\\ {\small  Xuzhou, Jiangsu 221116,
			People's Republic
			of China.}\\
		{\small E-mails: JXZhong@163.com, luyong@jsnu.edu.cn}}
	
	\date{}
	\maketitle
	\begin{center}
		\begin{minipage}{120mm}
			\vskip 0.3cm
			\begin{center}
				{\small {\bf Abstract}}
			\end{center}
			{\small
				A graph $G$ is called $k$-factor-critical if after deleting any $k$ vertices the remaining subgraph still has a perfect matching. Fan and Lin [Adv. in Appl. Math. 174 (2026) 103019] posed an adjacency spectral condition for a graph with minimum degree to be $k$-factor-critical. A graph $G$ is fractional $k$-factor-critical if after deleting any $k$ vertices the remaining subgraph still has a fractional perfect matching. Clearly, the fractional $k$-factor-criticality of a graph is a necessary property for a graph to be $k$-factor-critical. Jia, Fan and Liu [Discrete Appl. Math. 386 (2026) 255-263] proposed a tight sufficient condition in terms of the spectral radius for a graph with fractional $k$-factor-criticality to be $k$-factor-critical. A natural question arises: can we derive analogous sufficient conditions by incorporating the minimum degree parameter of graphs?

				We first establish a lower bound on the size to ensure that a $(k+1)$-connected graph with fractional $k$-factor-criticality is $k$-factor-critical, where \(k\) is a positive integer with \(k \ge 1\). Moreover, we provide a sufficient condition in terms of the spectral radius for a $(k+1)$-connected graph with fractional $k$-factor-criticality to be $k$-factor-critical. Our results generalize the result of
				Jia, Fan and Liu to $(k+1)$-connected graphs. Furthermore, our spectral conditions apply to a broader family of connected graphs compared with the results of Fan and Lin, as well as Jia et al.
				
				\vskip 0.1in \noindent {\bf Keywords}:\ Spectral radius; Size; $k$-factor-critical; minimum degree; fractional $k$-factor-criticality. \vskip
				0.1in \noindent {\bf AMS Subject Classification (2020)}: \ 05C35; 05C50. }
		\end{minipage}
	\end{center}

	\section{Introduction }
	\hspace{1.3em}
	Let $G=(V(G),E(G))$ be a finite, undirected and simple graph, where $V(G)$ is the vertex set and $E(G)$ is the edge set. We denote the \emph{order} and \emph{size} of $G$ by $|V(G)|$ and $|E(G)|$, respectively. Let $\delta(G)$ be the \emph{minimum degree} (or simply $\delta$) of $G$.
	For a vertex subset $S$ of $G$, we denote by $G-S$ and $G[S]$ the subgraph of $G$ obtained from $G$ by deleting the vertices in $S$ together with their incident edges and the subgraph of $G$ induced by $S$, respectively. The number of components and the number of odd components  of $G$ are denoted by $c(G)$ and $o(G)$, respectively. The set of isolated vertices of $G$ is denoted by $i(G)$. Let $K_{n}$ denote the \emph{complete graph} of order $n$.
	For two vertex-disjoint graphs $G_{1}$ and $G_{2}$, we use $G_{1}\cup G_{2}$ to denote the \emph{disjoint union} of $G_{1}$ and $G_{2}$.
	The \emph{join} $G_{1}\vee G_{2}$ is the graph obtained from $G_{1} \cup G_{2}$ by adding all possible edges between $V(G_{1})$ and $V(G_{2})$. A graph $G$ of order $n$ is called \emph{$k$-connected} if $n>k$ and $G-X$ is connected for every set $X\subseteq V(G)$ with $|X|<k$, where $k$ is a positive integer.
	
	For a simple graph $G$ of order $n$, its \emph{adjacency matrix} $A(G)=(a_{ij})_{n\times n}$ is a symmetric matrix with $a_{ij}=1$ if and only if vertices $v_{i}$ and $v_{j}$ are adjacent, and $a_{ij}=0$ otherwise. The \emph{spectral radius} $\rho(G)$ of $G$ is the largest eigenvalue of $A(G)$.

	A \emph{matching} in a graph $G$ is a subset $M$ of $E(G)$ such that no vertex of $G$ is incident with more than one edge in $M$. A \emph{perfect matching} in $G$ is a matching such that every vertex of $G$ is incident with precisely one edge in it. The problem of determining when a general graph has a perfect matching is a classic topic in graph theory. In 2005, Brouwer and Haemers \cite{BH} provided several spectral sufficient conditions for the existence of a perfect matching in a graph. Motivated by their works, Fan et al. \cite{FSHL} gave a spectral radius condition for the existence of a perfect matching in a balanced bipartite graph with minimum degree $\delta$. Since then, many researchers have been interested in finding sufficient conditions to guarantee the existence of a perfect matching in a graph by using various graph invariants, see \cite{BY, EJK, FLL, LFS}.
	
	To enhance the understanding of graph structures with perfect matchings, Gallai \cite{G} and Lov\'{a}sz \cite{L} introduced the concepts of factor-critical and bicritical graphs, respectively. A graph $G$ is defined as \emph{factor-critical} (resp. \emph{bicritical}) if the removal of any vertex $v$ (resp. two vertices $u, v$) results in a graph $G-v$ (resp. $G-u-v$) that has a perfect matching. Motivated by the similarities and interesting properties of these two concepts, Favaron \cite{C} and Yu \cite{MK} independently introduced the concept of $k$-factor-critical graphs: a graph $G$ of order $n$ is \emph{$k$-factor-critical} (where $0 \leq k < n$) if removing any $k$ vertices leaves a graph with a perfect matching. The case $k=0$ is trivial and will not be discussed further in this article. Clearly, if $G$ is a $k$-factor-critical graph with $n$ vertices, then $k$ and $n$ have the same parity. Zhang and Fan \cite{ZF} provided sufficient conditions for a graph to be $k$-factor-critical in terms of the edge number and the spectral radius. Fan and Lin \cite{BY} further derived an adjacency spectral condition for a connected graph with minimum degree to be $k$-factor-critical. Motivated by these results, Zheng et al. \cite{ZLLW} established three sufficient conditions based on the size, the signless Laplacian spectral radius and the distance spectral radius of a connected graph with minimum degree to be $k$-factor-critical. Recently, Wang et al. \cite{MW} gave two sufficient conditions in terms of  the size and the spectral radius for $k$-factor-criticality of $t$-connected graphs. More results on the relationships between the spectral radius and spanning subgraphs can be found in \cite{FL, FDD, CFL, XZL, ZC, ZSZ}.

	A \emph{fractional perfect matching} of a graph $G$ is a function $h$ giving each edge a number in $[0,1]$ such that $\sum_{e\in E_G(v)} h(e) = 1$ for each $v\in V(G)$, where $E_G(v)$ is the set of edges incident to $v$. A graph $G$ is \emph{fractional $k$-factor-critical} if after deleting any $k$ vertices the remaining subgraph still has a fractional perfect matching.
	Obviously, if $G$ is $k$-factor-critical, then it is necessarily fractional $k$-factor-critical, whereas the converse is not universally valid.
	Naturally, the following problem was raised.
	
	\begin{problem}
		What is a tight spectral radius condition which guarantees a graph with fractional $k$-factor-criticality to be $k$-factor-critical?
	\end{problem}
	Concerning Problem 1.1, Jia, Fan and Liu \cite{JFL} proved the following result.

	\noindent\begin{theorem}[\cite{JFL}]\label{th:1.2.}
		Let \( G \) be a connected graph with fractional \( k \)-factor-criticality of order \( n \ge \max\{k+56,~ 23k+20\} \), where \( n \equiv k \pmod{2} \) and \( k \ge 1 \) is an integer. If
		\[
		\rho(G) \ge \rho(K_k \vee (K_{n-k-3} \cup K_3)),
		\]
		then \( G \) is \( k \)-factor-critical unless \( G \cong K_k \vee (K_{n-k-3} \cup K_3) \).
	\end{theorem}
	
	Motivated by the above result, our first main result gives
	a lower bound based on the size to ensure that a $(k+1)$-connected graph with fractional $k$-factor-criticality is $k$-factor-critical.
	
	\noindent\begin{theorem}\label{th:1.3.}
		Let \( G \) be a $(k+1)$-connected graph with fractional \( k \)-factor-criticality of order \( n \ge \max\{6\delta-5k+8,~ \frac{\delta^2}{6}-(\frac{k}{3}-\frac{3}{2})\delta+\frac{k^2}{6}-\frac{k}{2}+4\} \), where \( \delta \) is the minimum degree of \( G \), \( k \ge 1 \) is an integer and \( n \equiv k \pmod{2} \). If
		\[
		|E(G)| \ge |E(K_\delta \vee (K_{n-2\delta + k-3} \cup K_3 \cup (\delta-k)K_1))|,
		\]
		then \( G \) is \( k \)-factor-critical unless \( G \cong K_\delta \vee (K_{n-2\delta + k-3} \cup K_3 \cup (\delta-k)K_1) \).
	\end{theorem}
	
	Our second main result provides a sufficient condition in terms of the spectral radius for a $(k+1)$-connected graph with fractional $k$-factor-criticality to be $k$-factor-critical.
	
	\noindent\begin{theorem}\label{th:1.4.}
		Let \( G \) be a $(k+1)$-connected graph with fractional \( k \)-factor-criticality of order \( n \ge \max\{8\delta-5k+20,~ \delta(\delta-k-1)^2 - 2\delta-1\} \), where \( \delta \) is the minimum degree of \( G \), \( k \ge 1 \) is an integer and \( n \equiv k \pmod{2} \). If
		\[
		\rho(G) \ge \rho(K_\delta \vee (K_{n-2\delta + k-3} \cup K_3 \cup (\delta-k)K_1)),
		\]
		then \( G \) is \( k \)-factor-critical unless \( G \cong K_\delta \vee (K_{n-2\delta + k-3} \cup K_3 \cup (\delta-k)K_1) \).
	\end{theorem}
	
	One can check that $\rho(K_\delta \vee (K_{n-2\delta + k-3} \cup K_3 \cup (\delta-k)K_1)<\rho(K_k \vee (K_{n-k-3} \cup K_3))$ for $\delta\geq k + 1$ and $n\geq 2\delta -k + 6$.
	
	Indeed, Fan and Lin \cite{BY} posed an adjacency spectral condition for a connected graph with minimum degree to be $k$-factor-critical.
	
	\begin{theorem}[\cite{BY}]\label{th:1.5.}
		Suppose that $G$ is a connected graph of order $n \geq \max\{8\delta - 5k + 4,\, \delta(\delta - k)^2 + \delta - 1\}$ with minimum degree $\delta \geq k$, where $n \equiv k \pmod{2}$ and $k \geq 1$. If
		\[
		\rho(G) \geq \rho(K_\delta \vee (K_{n-2\delta+k-1} \cup (\delta - k + 1)K_1)),
		\]
		then $G$ is $k$-factor-critical unless $G \cong K_\delta \vee (K_{n-2\delta+k-1} \cup (\delta - k + 1)K_1)$.
	\end{theorem}
	
	We can also check that $\rho(K_\delta \vee (K_{n-2\delta + k-3} \cup K_3 \cup (\delta-k)K_1)<\rho(K_\delta \vee (K_{n-2\delta+k-1} \cup (\delta - k + 1)K_1))$ for $\delta\geq k + 1$ and $n\geq 2\delta - k + 6$.

	For clarity, the exceptional graph in Theorem \ref{th:1.3.} and Theorem \ref{th:1.4.} is shown in Figure 1.
	\begin{figure}[!htbp]
		\centering
		\includegraphics[scale=0.42]{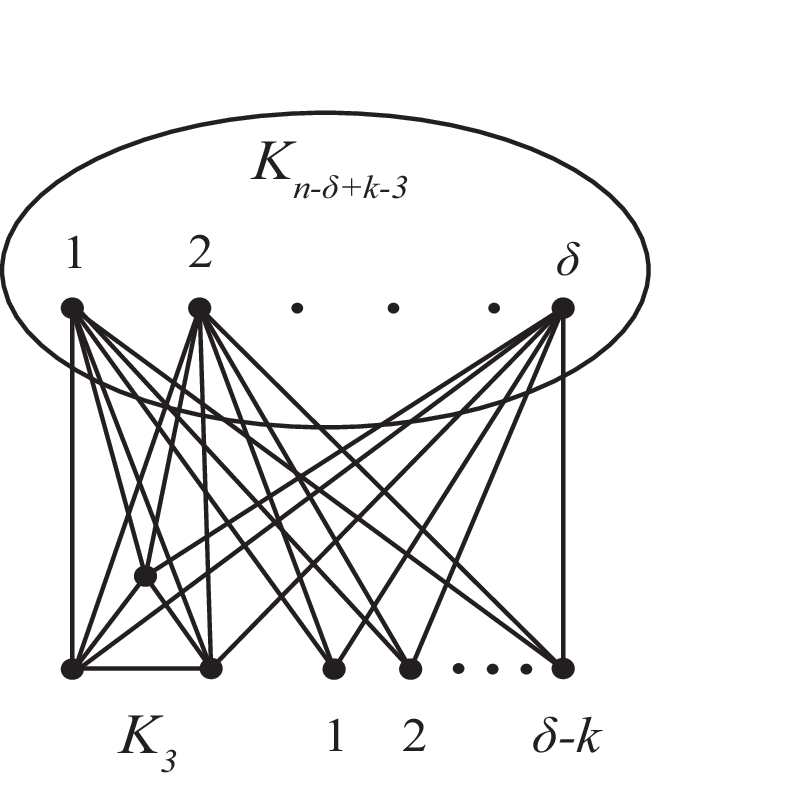}
		\caption{$K_\delta \vee (K_{n-2\delta + k-3} \cup K_3 \cup (\delta-k)K_1)$.}
	\end{figure}

	In the proof of Theorem 1.3 and Theorem 1.4, we assume that $G$ is $(k+1)$-connected, which leaves the question open of whether these results still hold for graphs with exactly connectivity $k$.
	
	The rest of this paper is organized as follows: In Section 2, some lemmas used in this paper are presented. In Section 3, we present the proof of Theorem 1.3. In Section 4, we present the proof of Theorem 1.4.
	
	\section{Preliminaries }
	\hspace{1.3em}
	In this section, we first present two fundamental characterizations that provide necessary and sufficient conditions for a graph to be $k$-factor-critical and fractional $k$-factor-critical, respectively.
	\noindent\begin{lemma}[\cite{C}]\label{le:2.1.}
		Let $G$ be a graph of order $n$ and $k\geq 1$ be an integer. Then $G$ is $k$-factor-critical if and only if
		$n\equiv k \pmod{2}$ and $o(G-S)\leq|S|-k$ for any $S\subseteq V(G)$ with $|S|\geq k$.
	\end{lemma}

	\noindent\begin{lemma}[\cite{M}]\label{le:2.2.}
		Let $G$ be a graph of order $n$ and $1\leq k\leq n-2$ be an integer. Then $G$ is fractional $k$-factor-critical if and only if
		$i(G-S)\leq|S|-k$ for any $S\subseteq V(G)$ with $|S|\geq k$.
	\end{lemma}
	
	\noindent\begin{lemma}[\cite{FDL}]\label{le:2.3.}
		Let $n=\sum_{i=1}^{t} n_i + s$ with $s\geq 1$. If $n_1\geq n_2 \geq \cdots \geq n_t\geq 1$, $ n_2 \geq 3$ and $n_1 < n-s-t-1$, then
		$$\rho(K_s\vee(K_{n_1}\cup K_{n_2}\cup \cdots \cup K_{n_t}))<
		\rho(K_s\vee (K_{n-s-t-1} \cup K_3 \cup (t-2)K_1)).$$
		
	\end{lemma}
	
	According to Lemma 2.3, we give the following critical lemma which generalizes the above lemma.
	
	\noindent\begin{lemma}\label{le:2.4.}
		Let $n=\sum_{i=1}^{t} n_i + s$ with $s\geq 1$. If $n_1\geq n_2\geq \cdots \geq n_t\geq p\geq 1$, $n_2 \geq 3$ and $n_1 < n-s-(t-2)p-3$, then
		$$\rho(K_s\vee(K_{n_1}\cup K_{n_2}\cup \cdots \cup K_{n_t}))<
		\rho(K_s\vee (K_{n-s-(t-2)p-3} \cup K_3 \cup (t-2)K_p)).$$	
	\end{lemma}
	\noindent\textbf{Proof.}
	
	Let $G=K_s\vee(K_{n_1}\cup K_{n_2}\cup \cdots \cup K_{n_t})$ and $G'= K_s\vee (K_{n-s-(t-2)p-3} \cup K_3 \cup (t-2)K_p)$. The vertex set of $G$ can be partitioned as $V(G) = V(K_s) \cup V(K_{n_1}) \cup V(K_{n_2}) \cup \cdots \cup V(K_{n_t})$, where $V(K_s)= \{v_1, v_2, \ldots, v_s\}$ and $ V(K_{n_i}) =\{u_{i1}, u_{i2}, \ldots, u_{in_i}\} $ for $ 1\leq i\leq t$. Let
	\begin{align*}
		E_1=&\{u_{ik}u_{il} \mid 3\leq i\leq t,\ 1\leq k\leq n_i - p,\ n_i - p + 1\leq l\leq n_i\}
		\\&\cup \{u_{2k}u_{2l} \mid  1\leq k\leq n_2 - 3,\ n_2 - 2\leq l\leq n_2\}, \\[2pt]
		E_2=&\{u_{1k}u_{ij} \mid 1\leq k\leq n_1,\ 3\leq i\leq t,\ 1\leq j\leq n_i - p\}
		\\&\cup \{u_{1k}u_{2j} \mid 1\leq k\leq n_1,
		\ 1\leq j\leq n_2 - 3\}, \\[2pt]
		E_3=&\{u_{ij}u_{kl} \mid 3\leq i\leq t-1,\ i+1\leq k\leq t;\ 1\leq j\leq n_i - p,\ 1\leq l\leq n_k - p\}
		\\& \cup \{u_{2j}u_{kl} \mid \ 3\leq k\leq t,\ 1\leq j\leq n_2 - 3,\ 1\leq l\leq n_k - p\}.
	\end{align*}
	It is straightforward to check that $G + E_2+ E_3 - E_1 \cong G'$.
	
	Let $ \mathbf{x} = (x_1, x_2,\ldots, x_n)^T $ be the Perron vector of $ A(G)$ corresponding to $\rho(G)$, and let $ x_i $ denote the entry of $ \mathbf{x} $ corresponding to the vertex $ v_i \in V(G) $. By symmetry, let $ x_r = y $ for all $ r\in V(K_s) $ and $ x_r =x_i $ for all $ r \in V(K_{n_i}) $, where $1\leq i\leq t$. Then, from $ A(G)\mathbf{x} = \rho(G)\mathbf{x} $, we get
	$$\rho(G)x_1 = sy + (n_1 - 1)x_1, ~\rho(G)x_j = sy + (n_j - 1)x_j~ (j=2, 3, \ldots , t).$$
	Then, we have
	$$[\rho(G)-(n_1 - 1)]x_1 =[\rho(G)-(n_j - 1)]x_j, ~2\leq j\leq t.$$
	Note that $K_{n_1}$ and $K_{n_j}$ are two proper subgraphs of $G$, we have $\rho(G)>\rho(K_{n_1})=n_1 - 1$ and $\rho(G)>\rho(K_{n_j})=n_j - 1$. Together with $n_1 \geq n_2\geq \cdots \geq n_t\geq p\geq 1$, we obtain $x_1\geq x_j$ for $2\leq j\leq t$.
	According to the Rayleigh quotient, we have
	\begin{align*}
		\rho(G') - \rho(G) &\ge \mathbf{x}^T (A(G') - A(G)) \mathbf{x} \\
		&= 2\Bigl[ n_1(n_2-3)x_1x_2 + \sum_{j=3}^{t} n_1(n_j-p)x_1x_j - 3(n_2-3)x_2^2 - \sum_{j=3}^{t} p(n_j-p)x_j^2 \\
		&\quad + \sum_{j=3}^{t} (n_j-p)(n_2-3)x_2x_j + \sum_{i=3}^{t-1} \sum_{j=i+1}^{t} (n_i-p)(n_j-p)x_ix_j \Bigr] \\
		&= 2\Bigl[ (n_2-3)(n_1x_1 - 3x_2)x_2 + \sum_{j=3}^{t} (n_j-p)(n_1x_1 - px_j)x_j \\
		&\quad + \sum_{j=3}^{t} (n_j-p)(n_2-3)x_2x_j + \sum_{i=3}^{t-1} \sum_{j=i+1}^{t} (n_i-p)(n_j-p)x_ix_j \Bigr] \\
		&> 0,
	\end{align*}
	since $x_1\ge x_j\ (2\le j\le t)$, $n_j\ge p\ge1\ (3\le j\le t)$ and $3\le n_2\le n_1 < n-s-(t-2)p-3$.
	
	Thus, the result follows.  \(\square\)

	\noindent\begin{lemma}\label{le:2.5.}
		Let \( n\geq 8\delta - 5k + 20 \), where \( k\geq 1 \) and \( \delta \geq k + 2 \). Then
		$$\rho(K_{k+1} \vee (K_{n-\delta-4} \cup K_{3} \cup K_{\delta-k}))
		<
		\rho(K_{\delta} \vee (K_{n-2\delta+k-3} \cup K_{3} \cup (\delta-k)K_1)).$$
	\end{lemma}
	\noindent \textbf{Proof.}
	
	Let \( G = K_{k+1} \vee (K_{n-\delta-4} \cup K_{3} \cup K_{\delta-k}) \). Then the vertex set of \( G \) can be partitioned as \( V(G) = V(K_{k+1}) \cup V(K_{n-\delta-4}) \cup V(K_{3}) \cup V(K_{\delta -k})\), where \( V(K_{k+1}) = \{v_1, v_2, \dots, v_{k+1}\} \), \( V(K_{n-\delta-4}) = \{u_1, u_2, \dots, u_{n-\delta-4}\} \), \( V(K_{3}) = \{z_1, z_2, z_3\} \) and \( V(K_{\delta -k}) = \{w_1, w_2, \dots, w_{\delta -k}\} \). Suppose that
	\[
	E_1 = \{u_i w_j \mid 1\leq i \le \delta-k-1,\ 1 \le j \le \delta-k\} \cup \{u_i z_j \mid 1\leq i \le \delta-k-1,\ 1 \le j \le 3\}
	\]
	and
	\[
	E_2 = \{w_i w_j \mid 1 \le i \le \delta-k-1,\ i+1 \le j \le \delta-k\}.
	\]
	Let \( G' = G + E_1 - E_2 \). Clearly, \( G' \cong K_{\delta} \vee (K_{n-2\delta+k-3} \cup K_{3} \cup (\delta-k)K_1) \).
	
	Let \( \mathbf{x} \) be the Perron vector of \( A(G) \) corresponding to $\rho(G)$. By symmetry, \( \mathbf{x} \) takes the same value (say \( x_1 \), \( x_2 \) , \( x_3 \) and \( x_4 \)) on the vertices of \( V(K_{k+1}),  V(K_{n-\delta-4}), V(K_{3}) ~and ~V(K_{\delta -k}) \), respectively. Then, by \( A(G)\mathbf{x} = \rho(G)\mathbf{x} \), we have
	\[
	\begin{cases}
		\rho x_2 = (k+1) x_1 + (n-\delta-5) x_2, \\
		\rho x_3 = (k+1) x_1 + 2 x_3,\\
		\rho x_4 = (k+1) x_1 + (\delta-k-1) x_4,
	\end{cases}
	\]
	which leads to
	\[
	(\rho - \delta + k + 1)(x_2 - x_4) = (n - 2\delta + k - 4)x_2 > 0,
	\]
	because \( n \ge 8\delta - 5k + 20 \), \( k \ge 1 \) and \( \delta \ge k+2 \). Note that \( K_{n-\delta+k-3} \) is a proper subgraph of \( G \), then \( \rho(G) > n - \delta + k - 4 > \delta - k - 1 \). Therefore,
	\[
	x_2 > x_4,
	\]
	and hence
	\begin{align*}
		\rho(G') - \rho(G) &\ge\mathbf{x}^T (A(G') - A(G)) \mathbf{x}\\
		&= 2\left(
		\sum_{j=1}^{\delta-k} \sum_{i=1}^{\delta-k-1} x_{u_i} x_{w_j} + \sum_{j=1}^{3} \sum_{i=1}^{\delta-k-1} x_{u_i} x_{z_j}
		- \sum_{i=1}^{\delta-k-1} \sum_{j=i+1}^{\delta-k} x_{w_i} x_{w_j}
		\right) \\
		&= 2[ (\delta-k)(\delta-k-1)x_2x_4 + 3(\delta-k-1)x_2x_3 - \frac{(\delta-k)(\delta-k-1)}{2}x_4^2 ] \\
		&= 2[x_4 (\delta-k)(\delta-k-1) ( x_2 - \frac{x_4}{2} ) + 3(\delta-k-1)x_2x_3] \\
		&> 0.
	\end{align*}
	
	The result follows. \(\square\)
	
	Above, we compare the spectral radius of the two kinds of extremal graphs. Now we compare their edge numbers.
	
	\noindent\begin{lemma}\label{le:2.6.}
		Let $n=\sum_{i=1}^{t} n_i + s$ with $s\geq 1$. If $n_1\geq n_2\geq \cdots \geq n_t\geq p\geq 1$, $n_2 \geq 3$ and $n_1 < n-s-(t-2)p-3$, then
		$$|E(K_s\vee(K_{n_1}\cup K_{n_2}\cup \cdots \cup K_{n_t}))|<
		|E(K_s\vee (K_{n-s-(t-2)p-3} \cup K_3 \cup (t-2)K_p))|.$$	
	\end{lemma}
	\noindent\textbf{Proof.}
	
	Let $G=K_s\vee(K_{n_1}\cup K_{n_2}\cup \cdots \cup K_{n_t})$ and $G'= K_s\vee (K_{n-s-(t-2)p-3} \cup K_3 \cup (t-2)K_p)$. The vertex set of $G$ can be partitioned as $V(G) = V(K_s) \cup V(K_{n_1}) \cup V(K_{n_2}) \cup \cdots \cup V(K_{n_t})$, where $V(K_s)= \{v_1, v_2, \ldots, v_s\}$ and $ V(K_{n_i}) =\{u_{i1}, u_{i2}, \ldots, u_{in_i}\} $ for $ 1\leq i\leq t$. Let
	\begin{align*}
		E_1=&\{u_{ik}u_{il} \mid 3\leq i\leq t,\ 1\leq k\leq n_i - p,\ n_i - p + 1\leq l\leq n_i\}
		\\&\cup \{u_{2k}u_{2l} \mid  1\leq k\leq n_2 - 3,\ n_2 - 2\leq l\leq n_2\}, \\[2pt]
		E_2=&\{u_{1k}u_{ij} \mid 1\leq k\leq n_1,\ 3\leq i\leq t,\ 1\leq j\leq n_i - p\}
		\\&\cup \{u_{1k}u_{2j} \mid 1\leq k\leq n_1,
		\ 1\leq j\leq n_2 - 3\}, \\[2pt]
		E_3=&\{u_{ij}u_{kl} \mid 3\leq i\leq t-1,\ i+1\leq k\leq t;\ 1\leq j\leq n_i - p,\ 1\leq l\leq n_k - p\}
		\\& \cup \{u_{2j}u_{kl} \mid \ 3\leq k\leq t,\ 1\leq j\leq n_2 - 3,\ 1\leq l\leq n_k - p\}.
	\end{align*}
	It is straightforward to check that $G + E_2+ E_3 - E_1 \cong G'$.
	Then, we have
	\begin{align*}
		|E(G')| - |E(G)|
		&=  n_1(n_2-3) + \sum_{j=3}^{t} n_1(n_j-p) - 3(n_2-3) - \sum_{j=3}^{t} p(n_j-p) \\
		&\quad + \sum_{j=3}^{t} (n_j-p)(n_2-3) + \sum_{i=3}^{t-1} \sum_{j=i+1}^{t} (n_i-p)(n_j-p)  \\
		&= (n_2-3)(n_1 - 3) + \sum_{j=3}^{t} (n_j-p)(n_1 - p) \\
		&\quad + \sum_{j=3}^{t} (n_j-p)(n_2-3) + \sum_{i=3}^{t-1} \sum_{j=i+1}^{t} (n_i-p)(n_j-p)  \\
		&> 0.
	\end{align*}
	due to $n_j \geq p\geq 1$, for $3\leq j\leq t$ and $3\leq n_2\leq n_1 < n-s-(t-2)p-3$.
	
	Thus, the result follows.  \(\square\)

	\noindent\begin{lemma}\label{le:2.7.}
		Let \( n\geq 6\delta - 5k + 8 \), where \( k\geq 1 \) and \( \delta \geq k + 2 \). Then
		$$|E(K_{k+1} \vee (K_{n-\delta-4} \cup K_{3} \cup K_{\delta-k}))|
		<
		|E(K_{\delta} \vee (K_{n-2\delta+k-3} \cup K_{3} \cup (\delta-k)K_1))|.$$
	\end{lemma}
	\noindent \textbf{Proof.}
	
	Let \( G = K_{k+1} \vee (K_{n-\delta-4} \cup K_{3} \cup K_{\delta-k}) \). Then the vertex set of \( G \) can be partitioned as \( V(G) = V(K_{k+1}) \cup V(K_{n-\delta-4}) \cup V(K_{3}) \cup V(K_{\delta -k})\), where \( V(K_{k+1}) = \{v_1, v_2, \dots, v_{k+1}\} \), \( V(K_{n-\delta-4}) = \{u_1, u_2, \dots, u_{n-\delta-4}\} \), \( V(K_{3}) = \{z_1, z_2, z_3\} \) and \( V(K_{\delta -k}) = \{w_1, w_2, \dots, w_{\delta -k}\} \). Suppose that
	\[
	E_1 = \{u_i w_j \mid 1\leq i \le \delta-k-1,\ 1 \le j \le \delta-k\} \cup \{u_i z_j \mid 1\leq i \le \delta-k-1,\ 1 \le j \le 3\}
	\]
	and
	\[
	E_2 = \{w_i w_j \mid 1 \le i \le \delta-k-1,\ i+1 \le j \le \delta-k\}.
	\]
	Let \( G' = G + E_1 - E_2 \). Clearly, \( G' \cong K_{\delta} \vee (K_{n-2\delta+k-3} \cup K_{3} \cup (\delta-k)K_1) \).
	Then, we have
	\begin{align*}
		|E(G')| - |E(G)|
		&= (\delta-k)(\delta-k-1) + 3(\delta-k-1) - \frac{(\delta-k)(\delta-k-1)}{2}  \\
		&=  \frac{ (\delta-k)(\delta-k-1)}{2} + 3(\delta-k-1) \\
		&> 0.
	\end{align*}
	
	The result follows. \(\square\)

	\noindent\begin{lemma}\label{le:2.8.}
		
		Let \( k \ge 1 \) and \( \delta \ge k+1 \) be integers, and let \( n - k \) be a positive even integer. Then the graph \( K_\delta \vee ( K_{n-2\delta+k-3} \cup K_3 \cup (\delta - k)K_1 ) \) is not \( k \)-factor-critical.

	\end{lemma}
	\noindent \textbf{Proof.}
	
	Let \( G = K_\delta \vee ( K_{n-2\delta+k-3} \cup K_3 \cup (\delta - k)K_1 ) \). Suppose to the contrary that \( G \) is \( k \)-factor-critical. By Lemma 2.1, for every subset \( S \subseteq V(G) \), we have \( o(G - S) \le |S| - k \). However, if we take \( S = V(K_\delta) \), then \( o(G - S) = \delta - k + 2 > \delta - k = |S| - k \), a contradiction.

	\section{Proof of Theorem \ref{th:1.3.}}
	\hspace{1.3em}
	
	In this section, we will give the proof of Theorem 1.3.\\
	
	\noindent\textbf{Proof of Theorem \ref{th:1.3.}.}
	
	Let \( G \) be a $(k+1)$-connected graph of order \( n \equiv k \pmod{2} \), which is fractional \( k \)-factor-critical. Since $G$ is $(k+1)$-connected, its minimum degree satisfies $\delta\geq k+1$.  Suppose to the contrary that \( G \) is not \( k \)-factor-critical. According to Lemma 2.1, there exists some subset \( S \subseteq V(G) \) with $|S|\geq k$ such that \( o(G - S) \ge |S| - k + 1 \). Since \( n \equiv k \equiv s+ o(G - S) \pmod{2} \), it follows that $s-k\equiv o(G - S) \pmod{2}$. Hence, we have \( o(G - S) \ge |S| - k + 2 \).
	
	Let \( |S| = s \) and \( o(G - S) = r \). Meanwhile, let \( P_1, P_2, \dots, P_r \) be the odd components and \( Q_1, Q_2, \dots, Q_t \) be the even components of \( G - S \) (even components may not exist), where \( |P_i| = p_i \) for \( 1 \le i \le r \) and \( |Q_j| = q_j \) for \( 1 \le j \le t \). Without loss of generality, we assume that \( p_1 \le p_2 \le \cdots \le p_r \), where \( r \ge s - k + 2 \).
	
	We claim that \( p_{s-k+1} \ge 3 \). Otherwise, \( p_1 = p_2 = \cdots = p_{s-k+1} = 1 \). Then \( i(G - S) = s - k + 1 \). Note that \( G \) is fractional \( k \)-factor-critical. By Lemma 2.2, we have \( i(G - S) \le s - k \) for any \( S \subseteq V(G) \) with \( s \ge k \), a contradiction. This implies that \( p_i \ge 3 \) for \( s - k + 1 \le i \le r \) and \( p_j \ge 1 \) for \( 1 \le j \le s - k \). Let \( n_1 = \sum_{i=s-k+2}^{r} p_i + \sum_{j=1}^{t} q_j \), \( n_2 = p_{s-k+1}, \dots, n_{s-k+2} = p_1 \). That is to say, \( n_i \ge 3 \) and \( n_j \ge 1 \) for \( 1 \le i \le 2 \) and \( 3 \le j \le s - k + 2 \), where \( n_1 \ge n_2 \ge \cdots \ge n_{s-k+2} \) are positive odd integers with \( \sum_{i=1}^{s-k+2} n_i = n - s \). Hence, we have \( n \ge 2s - k + 6 \). Then \( G \) is a spanning subgraph of \( G' = K_s \vee (K_{n_1} \cup K_{n_2} \cup \cdots \cup K_{n_{s-k+2}}) \). So we have
	\[
	|E(G)| \le |E(G')|,
	\]
	where equality holds if and only if \( G \cong G' \).
	If $s=k$, then $o(G-S)\geq s-k+2=2$, contradicting the $(k+1)$-connectivity of $G$. Thus, $s\geq k+1$. Next, we divide the proof into the following three possible cases according to the value of $s$.\\
	
	\textbf{Case 1.} $s\geq\delta+1$.
	
	Let \( G'' = K_s \vee (K_{n-2s+k-3} \cup K_{3} \cup (s-k)K_1) \). Note that \( s \ge \delta+1 \). Then Lemma 2.6 gives that
	\[
	|E(G')| \le |E(G'')|,
	\]
	where the equality holds if and only if \( (n_1, n_2, n_3, \dots, n_{s-k+2}) = (n-2s+k-3, 3, 1, \dots, 1) \).
	The vertex set of \( G'' \) can be partitioned as \( V(G'') = V(K_s) \cup V(K_{n-2s+k-3}) \cup V(K_{3}) \cup V((s-k)K_1)\), where \( V(K_s) = \{v_1, v_2, \dots, v_s\} \), \( V((s-k)K_1) = \{u_1, u_2, \dots, u_{s-k}\} \) , \( V(K_{n-2s+k-3}) = \{w_1, w_2, \dots, w_{n-2s+k-3}\} \) and \( V(K_3) = \{z_1, z_2, z_3\} \). Suppose that
	\[
	\begin{aligned}
		E_1 =& \{u_i w_j \mid \delta-k+1\leq i \le s-k,\ 1 \le j \le n-2s+k-3\}
		\\&\cup \{u_i u_j \mid \delta-k+1 \le i \le s-k-1,\ i+1 \le j \le s-k\},\\
		E_2 =& \{v_i u_j \mid \delta+1 \le i \le s,\ 1 \le j \le \delta-k\}\cup \{v_i z_j \mid \delta+1 \le i \le s,\ 1 \le j \le 3\}.
	\end{aligned}
	\]
	Let \( G_{\delta} = G'' + E_1 - E_2 \). Clearly, \( G_{\delta} \cong  K_\delta \vee (K_{n-2\delta+k-3} \cup K_{3} \cup (\delta-k)K_1)\). Since $n\geq2s - k + 6$ and $n\geq6\delta-5k+8$, we have
	\begin{align*}
		|E(G_{\delta})| - |E(G'')|
		&= (s-\delta)(n-2s+k-3) + \frac{(s-\delta)(s-\delta-1)}{2} - (s-\delta)(\delta-k)-3(s-\delta) \\
		&= \dfrac{1}{2}(s-\delta)(2n-3s-3\delta+4k-13) \\
		&\geq\dfrac{1}{2}(s-\delta)(2n-\frac{3(n+k-6)}{2}-3\delta+4k-13)\\
		&=\dfrac{1}{4}(s-\delta)(n-6\delta+5k-8)\\
		&>0.
	\end{align*}
	Then, we have $|E(G_{\delta})| > |E(G'')|$. Together with $|E(G)| \le |E(G')| \le |E(G'')|$, we obtain $|E(G)| < |E(G_{\delta})|$, a contradiction to the condition.

	\textbf{Case 2.} $s=\delta$.
	
	By Lemma 2.6, we have \(|E(G')| \le |E(K_\delta \vee (K_{n-2\delta+k-3} \cup K_{3} \cup (\delta-k)K_1))|\), with equality if and only if \(G' \cong K_\delta \vee (K_{n-2\delta+k-3} \cup K_{3} \cup (\delta-k)K_1)\). Note that \(|E(G)| \le |E(G')|\). Hence, we have \(|E(G)| \le |E(K_\delta \vee (K_{n-2\delta+k-3} \cup K_{3} \cup (\delta-k)K_1))|\), with equality if and only if \(G \cong K_\delta \vee (K_{n-2\delta+k-3} \cup K_{3} \cup (\delta-k)K_1)\). By Lemma 2.8, we have \(K_\delta \vee (K_{n-2\delta+k-3} \cup K_{3} \cup (\delta-k)K_1)\) is not \(k\)-factor-critical. Thus, we obtain a contradiction.

	\textbf{Case 3.} $k+1\leq s\leq\delta-1$.
	
	Let \( G^* = K_s \vee (K_{n-s-(s-k)(\delta+1-s)-3} \cup K_{3} \cup (s-k)K_{\delta+1-s}) \). Recall that \( G \) is a spanning subgraph of \( G' = K_s \vee (K_{n_1} \cup K_{n_2} \cup \cdots \cup K_{n_{s-k+2}}) \), where \( n_i \ge 3 \), \( n_j \ge 1 \) for \( 1 \le i \le 2 \), \( 3 \le j \le s - k + 2 \), \( n_1 \ge n_2 \ge \cdots \ge n_{s-k+2} \) and \( \sum_{i=1}^{s-k+2} n_i = n - s \). Clearly, \( n_{s-k+2} \ge \delta + 1 - s \) because the minimum degree of \( G' \) is at least \( \delta \). By Lemma 2.6, we have
	\[
	|E(G')| \le |E(G^*)|,
	\]
	where the equality holds if and only if \( (n_1, n_2, n_3, \dots, n_s-k+2) = (n - s - (s - k)(\delta + 1 - s)-3, 3, \delta + 1 - s, \dots, \delta + 1 - s) \).
	
	If \( s = k+1 \), then \( G^* = K_{k+1} \vee (K_{n-\delta-4} \cup K_{3} \cup K_{\delta-k}) \). Combining Lemma 2.7, we get
	\[
	|E(G)| \le |E(G')| \le |E(G^*)| < |E(K_\delta \vee (K_{n-2\delta+k-3} \cup K_{3} \cup (\delta-k)K_1))|.
	\]
	Thus, we consider \( s \ge k+2 \) in the following. Let $G_{\delta}=K_\delta \vee (K_{n-2\delta+k-3} \cup K_{3} \cup (\delta-k)K_1)$. Then we have
	\[
	\begin{aligned}
		|E(G^*)| &= \binom{n-(s-k)(\delta+1-s)-3}{2}
		+ \binom{\delta+1-s}{2}(s-k)+ s(\delta+1-s)(s-k) \\
		&\quad + 3(s+1), \\[4pt]
		|E(G_{\delta})| &= \binom{n-\delta+k-3}{2}+\delta(\delta-k)+3\delta+3.
	\end{aligned}
	\]
	By a direct calculation, we have
	\[
	\begin{aligned}
		|E(G_{\delta})| - |E(G^*)| &= \frac{1}{2}(\delta-s)[s^3+(-2k-\delta-3)s^2+(2n+k^2+2k\delta+5k-\delta-5)s\\
		&\quad-2(k+1)n +\delta(-k^2 + k + 3)-2k^2+4k+13 ].
	\end{aligned}
	\]
	Let \(f(x)=x^3+(-2k-\delta-3)x^2+(2n+k^2+2k\delta+5k-\delta-5)x-2(k+1)n +\delta(-k^2 + k + 3)-2k^2+4k+13\) be a real function in \(x\) with \(x\in [k+2,\delta-1]\). Then the derivative function of \(f(x)\) is
	\[
	f'(x)=3x^2+2(-2k-\delta-3)x+2n+k^2+2k\delta+5k-\delta-5.
	\]
	Note that the symmetry axis of parabola \(f'(x)\) is \(\frac{2k+\delta+3}{3}\). Then we have
	\[
	f'(x)\ge f'\!\left(\frac{2k+\delta+3}{3}\right)=[2n-\frac{\delta^2}{3}+\delta(\frac{2}{3}k-3)-\frac{k^2}{3}+k-8]\ge 0,
	\]
	which implies \(f(x)\) is increasing in the interval \([k+2,\delta-1]\). So we obtain
	\[
	\begin{aligned}
		f(s) &\ge f(k+2)=2n+k-3\delta-1 \\
		&\ge 2(6\delta-5k+8)+k-3\delta-1 \\
		&\ge 9\delta-9k+15 \\
		&>0.
	\end{aligned}
	\]
	Thus, \(|E(G_{\delta})| > |E(G^*)|\). Together with \(|E(G)| \le |E(G')| \le |E(G^*)|\), we obtain \(|E(G)| < |E(G_{\delta})|\), a contradiction to the condition.
	
	This completes the proof. \(\square\)

	\section{Proof of Theorem \ref{th:1.4.}}
	\hspace{1.3em}
	
	In this section, we will give the proof of Theorem 1.4.\\
	
	\noindent\textbf{Proof of Theorem \ref{th:1.4.}.}
	
	Let \( G \) be a $(k+1)$-connected graph of order \( n \equiv k \pmod{2} \), which is fractional \( k \)-factor-critical. Since $G$ is $(k+1)$-connected, its minimum degree satisfies $\delta\geq k+1$.  Suppose to the contrary that \( G \) is not \( k \)-factor-critical. According to Lemma 2.1, there exists some subset \( S \subseteq V(G) \) with $|S|\geq k$ such that \( o(G - S) \ge |S| - k + 1 \). Since \( n \equiv k \equiv s+ o(G - S) \pmod{2} \), it follows that $s-k\equiv o(G - S) \pmod{2}$. Hence, we have \( o(G - S) \ge |S| - k + 2 \).
	
	Let \( |S| = s \) and \( o(G - S) = r \). Meanwhile, let \( P_1, P_2, \dots, P_r \) be the odd components and \( Q_1, Q_2, \dots, Q_t \) be the even components of \( G - S \) (even components may not exist), where \( |P_i| = p_i \) for \( 1 \le i \le r \) and \( |Q_j| = q_j \) for \( 1 \le j \le t \). Without loss of generality, we assume that \( p_1 \le p_2 \le \cdots \le p_r \), where \( r \ge s - k + 2 \).
	
	We claim that \( p_{s-k+1} \ge 3 \). Otherwise, \( p_1 = p_2 = \cdots = p_{s-k+1} = 1 \). Then \( i(G - S) = s - k + 1 \). Note that \( G \) is fractional \( k \)-factor-critical. By Lemma 2.2, we have \( i(G - S) \le s - k \) for any \( S \subseteq V(G) \) with \( s \ge k \), a contradiction. This implies that \( p_i \ge 3 \) for \( s - k + 1 \le i \le r \) and \( p_j \ge 1 \) for \( 1 \le j \le s - k \). Let \( n_1 = \sum_{i=s-k+2}^{r} p_i + \sum_{j=1}^{t} q_j \), \( n_2 = p_{s-k+1}, \dots, n_{s-k+2} = p_1 \). That is to say, \( n_i \ge 3 \) and \( n_j \ge 1 \) for \( 1 \le i \le 2 \) and \( 3 \le j \le s - k + 2 \), where \( n_1 \ge n_2 \ge \cdots \ge n_{s-k+2} \) are positive odd integers with \( \sum_{i=1}^{s-k+2} n_i = n - s \). Hence, we have \( n \ge 2s - k + 6 \). Then \( G \) is a spanning subgraph of \( G' = K_s \vee (K_{n_1} \cup K_{n_2} \cup \cdots \cup K_{n_{s-k+2}}) \). So we have
	\[
	\rho(G) \le \rho(G'), \tag{1}
	\]
	where equality holds if and only if \( G \cong G' \).
	If $s=k$, then $o(G-S)\geq s-k+2=2$, contradicting the $(k+1)$-connectivity of $G$. Thus, $s\geq k+1$. Next, we divide the proof into the following three possible cases according to the value of $s$.\\
	
	\textbf{Case 1.} $s\geq\delta+1$.
	
	Let \( G'' = K_s \vee (K_{n-2s+k-3} \cup K_{3} \cup (s-k)K_1) \). Note that \( s \ge \delta+1 \). Then Lemma 2.4 gives that
	\[
	\rho(G') \le \rho(G''), \tag{2}
	\]
	where the equality holds if and only if \( (n_1, n_2, n_3, \dots, n_{s-k+2}) = (n-2s+k-3, 3, 1, \dots, 1) \).
	The vertex set of \( G'' \) can be partitioned as \( V(G'') = V(K_s) \cup V(K_{n-2s+k-3}) \cup V(K_{3}) \cup V((s-k)K_1)\), where \( V(K_s) = \{v_1, v_2, \dots, v_s\} \), \( V((s-k)K_1) = \{u_1, u_2, \dots, u_{s-k}\} \) , \( V(K_{n-2s+k-3}) = \{w_1, w_2, \dots, w_{n-2s+k-3}\} \) and \( V(K_3) = \{z_1, z_2, z_3\} \). Suppose that
	\[
	\begin{aligned}
		E_1 =& \{u_i w_j \mid \delta-k+1\leq i \le s-k,\ 1 \le j \le n-2s+k-3\}
		\\&\cup \{u_i u_j \mid \delta-k+1 \le i \le s-k-1,\ i+1 \le j \le s-k\},\\
		E_2 =& \{v_i u_j \mid \delta+1 \le i \le s,\ 1 \le j \le \delta-k\}\cup \{v_i z_j \mid \delta+1 \le i \le s,\ 1 \le j \le 3\}.
	\end{aligned}
	\]
	Let \( G_{\delta} = G'' + E_1 - E_2 \). Clearly, \( G_{\delta} \cong  K_\delta \vee (K_{n-2\delta+k-3} \cup K_{3} \cup (\delta-k)K_1)\). Let \( \mathbf{x} \) be the Perron vector of \( A(G'') \), and let \( \rho = \rho(G'') \). By symmetry, \( \mathbf{x} \) takes the same value (say \( x_1, x_2, x_3 \) and \( x_4 \)) on the vertices of \( V(K_s) \), \( V((s-k)K_1) \), \( V(K_{n-2s+k-3}) \) and \( V(K_{3}) \), respectively. Then, by \( A(G'')\mathbf{x} = \rho\mathbf{x} \), we have
	\[
	\begin{cases}
		\rho x_2 = s x_1, \\
		\rho x_3 = s x_1 + (n-2s+k-4)x_3,\\
		\rho x_4 = s x_1 + 2 x_4.
	\end{cases}
	\]
	Observe that \( n \ge 2s-k+6 \). Then \( x_3 > x_4 > x_2 \) and
	\[
	x_2 = \frac{s x_1}{\rho},~ x_3 = \frac{s x_1}{\rho - (n-2s+k-4)}, ~x_4 = \frac{s x_1}{\rho - 2}. \tag{3}
	\]
	Similarly, let \( \mathbf{y} \) be the Perron vector of \( A(G_{\delta}) \), and let \( \rho(G_{\delta}) = \rho' \). By symmetry, \( \mathbf{y} \) takes the same value (say \( y_1, y_2, y_3 \) and \( y_4 \)) on the vertices of \( V(K_\delta) \), \( V((\delta-k)K_1) \), \( V(K_{n-2\delta+k-3}) \) and \( V(K_{3}) \), respectively. Then, by \( A(G_{\delta})\mathbf{y} = \rho'\mathbf{y} \), we have
	\[
	\begin{cases}
		\rho' y_2 = \delta y_1, \\
		\rho' y_3 = \delta y_1 + (n-2\delta+k-4)y_3,\\
		\rho' y_4 = \delta y_1 + 2 y_4.
	\end{cases}
	\]
	Note that \( G_{\delta} \) contains \( K_{n-\delta+k-3}  \) as a proper subgraph. Then \( \rho' > n-\delta+k-4 \). Combining with the above, we have
	\[
	y_2 = \frac{\delta y_1}{\rho'},~ y_3 = \frac{\delta y_1}{\rho' - (n-2\delta+k-4)}, ~y_4 = \frac{\delta y_1}{\rho' - 2}. \tag{4}
	\]
	Note that \( n \ge 2s-k+6 \) and $\delta\geq k + 1$. Then \( k + 2 \leq \delta+1 \le s \le (n+k-6)/2 \). As \( G'' \) and \( G_{\delta} \) are not complete graphs, we have \( \rho < n-1 \) and \( \rho' < n-1 \). Now we shall prove that \( \rho < \rho' \). Indeed, if \( \rho \ge \rho' \), then from (3) and (4), we obtain	
	\begin{align*}
		& \quad y^T(\rho' - \rho)x \\
		&= y^T(A(G_{\delta}) - A(G''))x \\
		&= \sum_{i=\delta-k+1}^{s-k} \sum_{j=1}^{n-2s+k-3} (x_{u_i}y_{w_j} + x_{w_j}y_{u_i})
		+ \sum_{i=\delta-k+1}^{s-k-1} \sum_{j=i+1}^{s-k} (x_{u_i}y_{u_j} + x_{u_j}y_{u_i}) \\
		& \quad - \sum_{i=\delta+1}^{s} \sum_{j=1}^{\delta-k} (x_{v_i}y_{u_j} + x_{u_j}y_{v_i}) - \sum_{i=\delta+1}^{s} \sum_{j=1}^{3} (x_{v_i}y_{z_j} + x_{z_j}y_{v_i}) \\
		&= (s-\delta)\left[(n-2s+k-3)(x_2y_3 + x_3y_3) + (s-\delta-1)x_2y_3 - (\delta-k)(x_1y_2 + x_2y_3)
		\right. \\
		& \quad \left.
		- 3(x_1y_4 + x_4y_3)
		\right] \\
		&= (s-\delta)\left[(n-2\delta-s+2k-4)x_2y_3 + (n-2s+k-3)x_3y_3 - (\delta-k)x_1y_2 - 3x_1y_4-3x_4y_3\right] \\
		&> (s-\delta)\left[(2n-2\delta-3s+3k-7)x_2y_3 - (\delta-k)x_1y_2-3x_1y_4-3x_4y_3\right] \quad (\text{since } x_3 > x_2) \\
		&= (s-\delta)x_1y_3\left[
		\frac{s}{\rho}(2n-2\delta-3s+3k-7)
		- (\delta-k)\frac{\rho'-(n-2\delta+k-4)}{\rho'}
		\right. \\
		& \quad \left.
		- \frac{3(\rho'-(n-2\delta+k-4))}{\rho'-2}
		- \frac{3s}{\rho-2}
		\right] \quad (\text{by (3) and (4)}) \\
		&> \frac{(s-\delta)x_1y_3}{\rho(\rho'-2)} \left[ (2n-2\delta-3s+3k-7)s(\rho'-2) - \rho(\delta-k+3)(\rho'-(n-2\delta+k-4)) - 3\rho s\right] \\
		& \quad (\text{since } \rho\geq \rho' ~and \ \rho' > \rho'-2) \\
		&= \frac{(s-\delta)x_1y_3}{\rho(\rho'-2)}
		\left[
		(2n-2\delta-3s+3k-7)s(\rho'-2)
		- \rho(\delta-k+2)(\rho'-(n-2\delta+k-4))
		\right. \\
		& \quad \left.
		- 3\rho s
		- \rho(\rho'-(n-2\delta+k-4))
		\right]\\
		&> \frac{(s-\delta)x_1y_3}{\rho(\rho'-2)}
		\left[s(
		(2n-2\delta-3s+3k-7)(\rho'-2)
		- \rho(\rho'-(n-2\delta+k-4))- 3\rho)
		\right. \\
		& \quad \left.
		- \rho(n-1-(n-2\delta+k-4))
		\right] \quad (\text{since } s\geq\delta+1\geq k+2 ~and \ \rho' < n-1)\\
		&= \frac{(s-\delta)x_1y_3}{\rho(\rho'-2)}
		\left[s(
		(2n-2\delta-3s+3k-7)(\rho'-2)
		- \rho(\rho'-(n-2\delta+k-7))
		\right. \\
		&\quad \left.
		- (2\delta - k + 3)\rho)
		\right] \\
		&\geq\frac{s(s-\delta)x_1y_3}{\rho(\rho'-2)}
		\left[
		(2n-2\delta-3s+3k-7)(\rho'-2)
		- \rho(\rho'-(n-2\delta+k-9))\right]\\
		&\quad (\text{since } s\geq\delta+1\geq k+2 )\\
		&=\frac{s(s-\delta)x_1y_3}{\rho(\rho'-2)}
		\left[
		(2n-2\delta-3s+3k-7)\rho'
		- \rho\rho'+\rho(n-2\delta+k-9)
		\right. \\
		&\quad \left.
		- 2(2n-2\delta-3s+3k-7)\right]\\
		&\geq\frac{s(s-\delta)x_1y_3}{\rho(\rho'-2)}
		\left[\rho'
		(3n-4\delta-3s+4k-16-\rho)- 2(2n-2\delta-3s+3k-7)
		\right]
		\quad (\text{since }\rho\geq\rho' )\\
		&>\frac{s(s-\delta)x_1y_3}{\rho(\rho'-2)}
		\left[\rho'
		(3n-4\delta-\frac{3(n+k-6)}{2}+4k-16-(n-1))- 2(2n-2\delta-3s+3k-7)
		\right]\\
		&\quad (\text{since }\rho<n-1~ and ~s\leq\dfrac{n+k-6}{2} )\\
		&=\frac{s(s-\delta)x_1y_3}{\rho(\rho'-2)}
		\left[
		(\dfrac{n}{2}-4\delta+\dfrac{5}{2}k-6)\rho'- 2(2n-2\delta-3s+3k-7)
		\right]\\
		&>\frac{s(s-\delta)x_1y_3}{\rho(\rho'-2)}
		\left[4(n-\delta + k-4)- 2(2n-2\delta-3(\delta+1)+3k-7)
		\right]\\
		&\quad (\text{since }n\geq8\delta-5k+20~ and ~\rho'>n-\delta+k-4 )\\
		&=\frac{s(s-\delta)x_1y_3}{\rho(\rho'-2)}
		(6\delta-2k+4) \\
		&> 0 \quad (\text{since }s\geq\delta+1\geq k+2),
	\end{align*}
	a contradiction. Therefore, we have \( \rho' > \rho \), and it follows from (1) and (2) that
	\[
	\rho(G) \le \rho(G') \le \rho(G'') < \rho(K_\delta \vee (K_{n-2\delta+k-3} \cup K_{3} \cup (\delta-k)K_1)).
	\]

	\textbf{Case 2.} $s=\delta$.
	
	By Lemma 2.4, we have
	\[
	\rho(G') \le \rho(K_\delta \vee (K_{n-2\delta+k-3} \cup K_{3} \cup (\delta-k)K_1)),
	\]
	with equality holding if and only if \( G' \cong K_\delta \vee (K_{n-2\delta+k-3} \cup K_{3} \cup (\delta-k)K_1) \). Combining this with (1), we conclude that
	\[
	\rho(G) \le \rho(K_\delta \vee (K_{n-2\delta+k-3} \cup K_{3} \cup (\delta-k)K_1)),
	\]
	where the equality holds if and only if \( G \cong K_\delta \vee (K_{n-2\delta+k-3} \cup K_{3} \cup (\delta-k)K_1) \). By Lemma 2.8, we have \( K_\delta \vee (K_{n-2\delta+k-3} \cup K_{3} \cup (\delta-k)K_1) \) is not \( k \)-factor-critical.  Thus the result follows. \\

	\textbf{Case 3.} $k+1\leq s\leq\delta-1$.
	
	Let \( G^* = K_s \vee (K_{n-s-(s-k)(\delta+1-s)-3} \cup K_{3} \cup (s-k)K_{\delta+1-s}) \). Recall that \( G \) is a spanning subgraph of \( G' = K_s \vee (K_{n_1} \cup K_{n_2} \cup \cdots \cup K_{n_{s-k+2}}) \), where \( n_i \ge 3 \), \( n_j \ge 1 \) for \( 1 \le i \le 2 \), \( 3 \le j \le s - k + 2 \), \( n_1 \ge n_2 \ge \cdots \ge n_{s-k+2} \) and \( \sum_{i=1}^{s-k+2} n_i = n - s \). Clearly, \( n_{s-k+2} \ge \delta + 1 - s \) because the minimum degree of \( G' \) is at least \( \delta \). By Lemma 2.4, we have
	\[
	\rho(G') \le \rho(G^*), \tag{5}
	\]
	where the equality holds if and only if \( (n_1, n_2, n_3, \dots, n_s-k+2) = (n - s - (s - k)(\delta + 1 - s)-3, 3, \delta + 1 - s, \dots, \delta + 1 - s) \).
	
	If \( s = k+1 \), then \( G^* = K_{k+1} \vee (K_{n-\delta-4} \cup K_{3} \cup K_{\delta-k}) \). Combining Lemma 2.5, (1) and (5), we get
	\[
	\rho(G) \le \rho(G') \le \rho(G^*) < \rho(K_\delta \vee (K_{n-2\delta+k-3} \cup K_{3} \cup (\delta-k)K_1)).
	\]
	Thus, we consider \( s \ge k+2 \) in the following. Assume that \( \rho(G^*) = \rho^* \ge n - (s - k)(\delta + 1 - s)-3 \). Let \( x \) be the Perron vector of \( A(G^*) \). By symmetry, \( x \) takes the same values \( x_1, x_2, x_3 \), and \( x_4 \) on the vertices of \( V(K_s) \), \( V(K_{n-s-(s-k)(\delta+1-s)-3}) \), \( V(K_{3}) \) and \( V(K_{\delta+1-s}) \), respectively. Then from \( A(G^*)x = \rho^*x \) we get
	\[
	\begin{cases}
		\rho^* x_1= (s-1)x_1 + (n-s-(s-k)(\delta+1-s)-3)x_2 + 3x_3 + (s-k)(\delta+1-s)x_4,  \\
		\rho^* x_2= s x_1 + (n-s-(s-k)(\delta+1-s)-4)x_2,  \\
		\rho^* x_3= s x_1 + 2x_3,\\
		\rho^* x_4= s x_1 + (\delta - s)x_4.
	\end{cases}
	\]
	Combining with the above, we yield that
	\[
	\begin{cases}
		x_2 = \dfrac{s x_1}{\rho^* - (n-s-(s-k)(\delta+1-s)-4)}, \\[1em]
		x_3 = \dfrac{s x_1}{\rho^* - 2},\\[1em]
		x_4 = \dfrac{s x_1}{\rho^* - \delta+s}.
	\end{cases} \tag{6}
	\]
	Note that \( n \geq  \delta(\delta - k-1)^2 -2\delta - 1 \), \( s \ge k + 2 \) and \( \delta \ge s + 1 \). Then \( \rho^* \ge n-(s-k)(\delta+1-s)-3 > \delta + 1 \). Combining this with (6), we have
	\begin{align*}
		\rho^* + 1
		&= s + (n-s-(s-k)(\delta+1-s)-3)\frac{x_2}{x_1} + \frac{3x_3}{x_1} + (s-k)(\delta+1-s)\frac{x_4}{x_1}\\
		&= s
		+ \frac{s(n - s - (s-k)(\delta+1-s)-3)}{\rho^* - (n - s  -(s-k)(\delta+1-s)-4)}+ \frac{3s}{\rho^* - 2}+\frac{s(\delta+1-s)(s-k)}{\rho^* - \delta + s} \\
		&< s + \frac{s(n-s)}{s+1} \quad (\text{since } \rho^* \ge n-(s-k)(\delta+1-s)-3 > \delta + 1) \\
		&= n - (s-k)(\delta+1-s)-3 - \frac{n - s - ((s-k)(\delta+1-s)-3)(s+1)}{s+1} \\
		&\le n -(s-k)(\delta+1-s)-3 - \frac{\delta((\delta-k-1)^2 - 3) -((s-k)(\delta+1-s)-3)(s+1) }{s+1} \\
		& \quad - \frac{\delta-s-1}{s+1}
		\quad (\text{since } n \ge \delta(\delta - k-1)^2 -2\delta - 1) \\
		&\le n -(s-k)(\delta+1-s)-3 \quad (\text{since } \delta \ge s+1 \text{ and } s \ge k+2) \\
		&\le \rho^*,
	\end{align*}
	which is impossible. Therefore,
	\begin{align*}
		\rho^* &< n - (s-k)(\delta+1-s)-3 \\
		&= n - \delta + k - 4 - ((s-k-1)(\delta-s)-1) \\
		&\le n - \delta + k - 4 \quad (\text{since } \delta \ge s+1 \text{ and } s \ge k+2).
	\end{align*}
	Since \( K_\delta \vee (K_{n-2\delta+k-3} \cup K_{3} \cup (\delta-k)K_1) \) contains \( K_{n-\delta+k-3} \) as a proper subgraph, we have \( \rho(K_\delta \vee (K_{n-2\delta+k-3} \cup K_{3} \cup (\delta-k)K_1)) > n - \delta + k - 4 \). Combining this with (1) and (5), we may conclude that
	\[
	\rho(G) \le \rho(G') \le \rho(G^*) < \rho(K_\delta \vee (K_{n-2\delta+k-3} \cup K_{3} \cup (\delta-k)K_1)).
	\]
	This completes the proof.
	\hfill\(\square\)\\

	\textbf{Declaration of competing interest}
	
	The authors declare that they have no known competing financial interests or personal relationships that could have appeared to influence the work reported in this paper.

	\textbf{Data availability}
	
	No data was used for the research described in the article.

\end{document}